\documentclass[12pt]{amsart}
\usepackage{amssymb}
\usepackage{verbatim}
\usepackage[usenames]{color}
\usepackage{url}
\usepackage{hyperref}
\newtheorem{theorem}{Theorem}

\newtheorem{lemma}[theorem]{Lemma}
\newtheorem{corollary}[theorem]{Corollary}

\theoremstyle{definition}
\newtheorem{definition}[theorem]{Definition}

\theoremstyle{remark}

\DeclareMathOperator{\Heads}{Heads} 
\DeclareMathOperator{\Min}{Min}     
\DeclareMathOperator{\Tails}{Tails} 

\newcommand{\A}{\smallskip\noindent\texttt A:\  }

\newcommand{\eps}{\varepsilon}
\newcommand{\F}{\mathcal F}

\newcommand{\Range}[1]{\ensuremath{{\text{Range}}(#1)}}
\newcommand{\Q}{\smallskip\noindent\texttt Q:\  }

\renewcommand{\P}{\mathbf P}
\renewcommand{\phi}{\varphi}

\title[p-values]{Fundamentals of p-values: Introduction}
\author{Yuri Gurevich}
\author{Vladimir Vovk}

\begin{document}

\begin{abstract}
  We explain the concept of p-values presupposing only rudimentary probability theory. We also use the occasion to introduce the notion of p-function, so that p-values are values of a p-function.

  The explanation is restricted to the discrete case with no outcomes of zero probability. We are going to address the general case elsewhere.
\end{abstract}

\maketitle

\section{Prelude}
\label{sec:prelude}

\noindent\texttt Q\footnotemark: The Cournot Principle that one of you wrote about \cite{G206} caught my attention:
``It is practically certain that a predicted event of small probability does not happen.''
How small should the probability be?

\footnotetext{Q and A are Quisani (a former student of the first author) and the authors respectively.}

\A Traditionally there have been two camps of naive Cournotians, the lax ones using the 5\% threshold and strict ones using the 1\% threshold.

\Q Why do you call them naive?

\A A better way is to report an appropriate p-value, without committing yourself to a threshold.

\Q Oh yes, I heard about p-values and even tried to read about them but got confused. Do you have a good reference?

\A Cox and Hinkley define p-values, even though they do not use the term,  in \S3 of the their 1974 book ``Theoretical Statistics'' \cite{CH}. That is the best reference that we have.

\Q How about explaining the concept to me right now? But please take into account that, while I have been exposed to probability theory and logic, I have not studied statistics.

\A The general case is somewhat involved \cite{GV}, but the discrete case is important and simple enough to explain the concept, especially if one presumes that every outcome has positive probability.

\section{Probability trials}
\label{sec:trials}

Recall that a discrete probability space is given by a nonempty set $\Omega$ and a function $\P$ from $\Omega$ to the real interval $[0,1]$ satisfying the following three requirements.
\begin{enumerate}
\item $\Omega$ is countable, i.e.\ finite or infinite countable.
\item Every $\P(x)\ge0$.
\item $\sum_{x\in\Omega} \P(x) = 1$.
\end{enumerate}
To simplify our exposition, we replace (2) with a stronger and relatively innocuous requirement
\begin{enumerate}
\item[(2$'$)] Every $\P(x) > 0$.
\end{enumerate}
The function $\P$ naturally extends to subsets $E$ of $\Omega$:\quad $\P(E) = \sum\{P(x): x\in E\}$.

Think of the probability space $(\Omega,\P)$ as a probability trial. Elements of $\Omega$ are outcomes, subsets of $\Omega$ are events, and $\P$ assigns probabilities to outcomes and events.
The probability distribution $\P$ is the \emph{null hypothesis} of the trial. Think of it as an alleged probability distribution. The purpose of the trial is to test $\P$.

\Q I guess, a predicted event of small probability, if and when it occurs, provides falsifying evidence against the null hypothesis.

\A Yes, but we do not need to define what probabilities count as small. For now, it suffices to say this:
The smaller the probability of the event in question, the stronger the impugning power of the event. Suppose for example that $\Omega = \{a,b,c\}$ and that $\P(a) = 990/1000$, $\P(b) = 9/1000$ and $\P(c) = 1/1000$. We can order the nonempty events in the order of increasing impugning power: $\{a,b,c\}$, $\{a,b\}$, $\{a,c\}$, $\{a\}$, $\{b,c\}$, $\{b\}$, $\{c\}$.

\Q But the event $\{a,b,c\}$ has zero impugning power.

\A True, and we don't say that it has any. We just say that its impugning power is less than that of the other 6 events. Also, we ignore the empty event because it does not occur.

\section{The logic angle}
\label{sec:logic}

\Q For logicians, a discrete probabilistic trial $(\Omega,\P)$ is a relatively simple logic structure, an extension of real arithmetic with a separate sort Outcome and a function $\P$ of type Outcome $\to$ Real, subject to requirements (1)--(3) of \S\ref{sec:trials}.

\A In applications, a trial comes with additional information.
Here are two examples, in a simplified form, from the paper you cited.

\medskip\noindent
\textit{Coin Example}.
  A coin is tossed 42 times. The null hypothesis is that all $2^{42}$ outcomes are equally probable. The actual outcome has 41 heads and just one tail. Clearly the impugning power of that outcome exceeds that of a random outcome.

\medskip\noindent
\textit{Lottery Example}.
4,000,000 people bought 10,000,000 tickets. Any of them could be the winner; thus these are 4,000,000 potential outcomes. The null hypothesis is that the probability of a person to win is proportional to the number of his or her tickets. The wife Donna of the lottery organizer John bought 3 tickets and won the lottery. Clearly the impugning power of that outcome exceeds that of a random outcome.

\Q Additional information complicates the matter. Clearly there are countless different kinds of additional information. How do statisticians deal with them?

\A They simplify the situation. They replace any additional structure with a linear preorder of the outcomes and restrict attention to the downward closed events.

\Q Wow, this is quite a simplification.

\A Often the linear preorder is given implicitly, by means of a test statistic or a nested family of events.

\Q Please remind me the necessary definitions.

\section{Test tools}
\label{sec:tools}

Let $(\Omega,\P)$ be a probability trial. A \emph{preorder} of a nonempty set $\Omega$ is a binary relation $\le$ on $\Omega$ that is reflexive and transitive so that
\begin{itemize}
  \item $x\le x$ for every $x$, and
  \item if $x\le y$ and $y\le z$ then $x\le z$.
\end{itemize}
A preorder $\le$ of $\Omega$ is \emph{linear} if
\begin{itemize}
  \item $x\le y$ or $y\le x$ for all $x,y$.
\end{itemize}

\begin{definition}\mbox{}
\begin{itemize}
  \item A \emph{test order} of $\Omega$ is a linear preorder on $\Omega$.
  \item A \emph{test pyramid} over $\Omega$ is a family of events linearly ordered by inclusion.
  \item A \emph{test statistic} on $\Omega$ is a function from $\Omega$ to the reals.
\end{itemize}
\end{definition}

\medskip\noindent\textit{Coin Example, cont}.
  Let $\Heads{(x)}$ and $\Tails{(x)}$ be the numbers of the heads and tails in an outcome $x$ respectively. The function $M(x) = \Min{\{\Heads{(x)},\Tails{(x)}\}}$ is a test statistic. The relation $M(x)\le M(y)$ is a linear preorder of the outcomes. Events $\{x: M(x) \le n\}$, where $n$ ranges from 0 to 42 form a test pyramid.
\medskip

Given a linear preorder $\le$ on $\Omega$, an event $E\subseteq\Omega$ is \emph{downward closed} if $x\le y\in E$ implies $x\in E$ for all $x,y$. The downward closed events are nested: if  $E_1, E_2$ are downward closed then either $E_1\subseteq E_2$ or $E_2\subseteq E_1$.
Indeed if $E_2\not\subseteq E_1$ and $x\in E_2-E_1$ then $E_1 \subseteq \{y: y\le x\} \subseteq E_2$.

Note that there may be distinct outcomes $x,y$ with $x\le y$ and $y\le x$; such outcomes $x,y$ called \emph{quasi-equivalent}. Linear preorders are also known as linear quasi-orders.

Test orders, test pyramids and test statistics are three tools to measure the relative strength of the impugning evidence provided by a given outcome or event. We adopt the following convention.
\begin{itemize}
\item[-]
For a given test order, smaller outcomes provide stronger impugning evidence.
\item[-]
For a given test statistic, smaller values provide stronger impugning evidence.
\end{itemize}
And of course,  for a given test pyramid of events, smaller events provide stronger impugning evidence.

In the following two sections we show that the three tools are equivalent in an appropriate sense.

\section{Test orders and test pyramids}
\label{sed:equiv}

If $\le$ is a test order, let $[\le x]$ be the downward closure $\{y: y\le x\}$ of $x$.

\begin{definition}\label{def:o-and-c}\mbox{}
\begin{itemize}
\item Every test order $\le$ \emph{induces} a test pyramid, namely the family of events $[\le x]$.
\item Every test pyramid $\F$ \emph{induces} a test order
\[
  \{ (x,y): (\forall E\in\F)[y\in E \implies x\in E] \}
\]
\end{itemize}
\end{definition}
In words, the formula $(\forall E\in\F)[y\in E \implies x\in E]$ says that every member of $\F$ that contains $y$ contains $x$ as well.

\begin{lemma}\label{lem:o-p-o}
If a test order $\le$ induces a test pyramid $\F$ then $\F$ induces $\le$.
\end{lemma}

\begin{proof}
Let $\preceq$ be the test order induced by $\F$.
\begin{align*}
x\preceq y
  &\iff (\forall E\in\F)[y\in E \implies x\in E]
    &&\text{by the definition of $x\preceq y$}\\
  &\iff (\forall z\in\Omega)[y\in[\le z] \implies x\in[\le z]]
    &&\text{by the definition of $\F$}\\
  &\iff x\in[\le y] \\
  &\iff x\le y
  &&\qedhere
\end{align*}
\end{proof}

\noindent
\begin{minipage}{\textwidth}
\begin{definition}\mbox{}
  \begin{itemize}
    \item Two test pyramids are \emph{equivalent} if they induce the same test order.
    \item The \emph{canonic version} $\hat{\F}$ of a test pyramid $\F$ is the test pyramid induced by the test order induced by $\F$.
    \item A test pyramid $\F$ is \emph{canonic} if it is the canonic version of itself.
  \end{itemize}
\end{definition}
\end{minipage}

\Q Show me a non-canonic test pyramid.

\A Let $\Omega$ be the set of numbers $\pm\frac1n$ where $n = 1, 2, \dots$. The precise definition of the distribution $\P$ on $\Omega$ is not important provided that every element of $\Omega$ has a positive probability. Let $\le$ be the standard order on the numbers $\pm\frac1n$. The desired test pyramid $\F$ is the collection of all subsets of $\Omega$ downward closed with respect to $\le$. Clearly $\F$ induces the test order $\le$. The pyramid $\hat\F$, induced by $\le$, consists of sets $[\le x]$ where $x\in\Omega$. It is a proper subcollection of $\F$. The difference $\F-\hat\F$ consists of three sets: $\emptyset$, $[\le0]$ and $\Omega$.

\Q You could have require that a test order $\le$ induces the pyramid of all the events downward closed with respect to $\le$. Then your counter-example would not work.

\A That is true. But then a non-canonic pyramid would be obtained from the pyramid of downward-closed events by omitting one or more of the sets $\emptyset$, $[\le0]$, $\Omega$.

\Q What about canonic test orders?

\A In our current setup every test order can be viewed as canonic. We saw already that every test order is induced by a test pyramid. We'll see below that every test order is induced by a test statistic. This is not so in the general case. In fact, it is not so even in the discrete case with outcomes of zero probability.

\begin{lemma}
Every pyramid $\F$ is equivalent to a unique canonic test pyramid, namely to $\hat{\F}$.
\end{lemma}

\begin{proof}
Let $\le$ be the test order induced by $\F$.
By the definition, $\hat{\F}$ is the collection of events $[\le x]$. Clearly $\hat{\F}$ induces $\le$, and so it is equivalent to $\F$.

It remains to check that $\hat{\F}$ is the only canonic test pyramid equivalent to $\F$. Let $\F'$ be a canonic test pyramid that is equivalent to $\F$ and therefore induces $\le$. Since $\F'$ is canonic, it is induced by $\le$. But then $\F' = \hat{\F}$.
\end{proof}

\begin{corollary}\label{cor:C2f}\mbox{}
\begin{enumerate}
\item If a test pyramid $\F$ induces a test order $\le$ then $\le$ induces the canonic version $\hat\F$ of $\F$.
\item If a test order $\le$ induces a test pyramid $\F$ then $\F$ is canonic.
\end{enumerate}
\end{corollary}

\begin{proof}\mbox{}
(1) By the definition of $\hat\F$.

\smallskip\noindent
(2) By Lemma~\ref{lem:o-p-o}, $\F$ induces $\le$. By (1), $\F =\hat\F$.
\end{proof}

\section{Test orders and test statistics}
\label{sec:o-s}

\begin{definition}\label{def:ind}\mbox{}
\begin{itemize}
\item Every test order $\le$ \emph{induces} a test statistic $f(x) = \P[\le x]$.
\item Every test statistic $f$ \emph{induces} a test order
$\{(x,y): f(x)\le f(y)\}$.
\end{itemize}
\end{definition}

For a test statistics $f$ and a real number $f$, let $[f\le r]$ be the event $\{x: f(x)\le r\}$. In the following lemma, we take advantage of having no zero-probability outcomes.

\begin{lemma}\label{lem:null1}
For any test order $\le$, we have
$x\le y$ if and only if $\P[\le x]\le \P[\le y]$.
\end{lemma}

\begin{proof}
If $x\le y$ then $[\le x] \subseteq [\le y]$ and therefore $\P[\le x] \le \P[\le y]$. This establishes the only-if implication. We prove the if implication by contrapositive. Suppose that $x>y$. Then $x$ belongs to $[\le x]$ but not to $[\le y]$ so that $[\le x]\supsetneq [\le y]$. Since $\P(x)>0$, we have $\P[\le x] > \P[\le y]$.
\end{proof}

\begin{lemma}\label{lem:null2}
For any test statistic $f$, we have:
\begin{enumerate}
\item If $f$ induces a test order $\le$ then every $[\le x] = [f\le f(x)]$.
\item $f(x)\le f(y) \iff \P[f\le f(x)] \le \P[f\le f(y)]$.
\end{enumerate}
\end{lemma}

\begin{proof}\mbox{}

\smallskip\noindent
(1) $[\le x] = \{y: y\le x\} = \{y: f(y)\le f(x)\} = [f\le f(x)]$.

\smallskip\noindent
(2) Follows from (1) and Lemma~\ref{lem:null1}.
\end{proof}

\begin{lemma}\label{lem:o-s-o}
If a test order $\le$ induces a test statistic $f$ then $f$ induces $\le$.
\end{lemma}

\begin{proof}
Let $\preceq$ be the test order induced by $f$.
\begin{align*}
x\le y
  &\iff \P[\le x] \le \P[\le y]
                      && \text{by Lemma~\ref{lem:null1},}\\
  &\iff f(x) \le f(y) && \text{by the definition of $f$,}\\
  &\iff x \preceq y   && \text{by the definition of $\preceq$.}
  \qedhere
\end{align*}
\end{proof}

\begin{definition}\label{def:norm}\mbox{}
\begin{itemize}
\item Two test statistics are \emph{equivalent} if they induce the same test order.
\item The \emph{canonic version} of a test statistic $f$ is the test statistic\\ $\hat f(x) = \P [f\le f(x)]$.
\item A test statistic is \emph{canonic} if it is the canonic version of itself.
\end{itemize}
\end{definition}

\begin{lemma}\label{lem:equiv}
Every test statistic $f$ is equivalent to a unique canonic test statistics, namely to $\hat f$.
\end{lemma}

\begin{proof}
First we check that $f$ is equivalent to $\hat{f}$.
\begin{align*}
f(x)\le f(y)
    &\iff \P[f\le f(x)] \le \P[f\le f(y)]
            &&\text{by Lemma~\ref{lem:null2},}\\
    &\iff \hat f(x)\le \hat f(y)
            &&\text{by the definition of $\hat{f}$}.
\end{align*}
Second we check that $\hat{f}$ is the only canonic test statistic equivalent to $f$. If $f'$ is a canonic test statistic equivalent to $f$ then
\begin{align*}
f'(x)\le f'(y)
    &\iff f(x)\le f(y)
            &&\text{because $f$ and $f'$ are equivalent,}\\
    &\iff \hat f(x)\le \hat f(y)
            &&\text{because $f$ and $\hat f$ are equivalent.}
\end{align*}
Thus $f'$, $\hat f$ and $f$ are all equivalent and therefore induce the same order $\le$.
\begin{align*}
 f'(x) &= \P[f'\le f'(x)]
           &&\text{as $f'$ is canonic}\\
       &= \P[\le x]
           &&\text{by Lemma~\ref{lem:null1}, part~1}\\
       &= \P[\hat f \le \hat f(x)]
           &&\text{by Lemma~\ref{lem:null1}, part~1}\\
       &= \hat f(x)
           &&\text{as $\hat f$ is canonic}.\qedhere
\end{align*}
\end{proof}

\begin{corollary}\label{cor:C2f}\mbox{}
\begin{itemize}
\item If a test statistic $f$ induces a test order $\le$ then $\le$ induces the canonic version $\hat f$ of $f$.
\item The test statistic induced by any test order is canonic.
\end{itemize}
\end{corollary}

\begin{proof}\mbox{}

\smallskip\noindent
(1) Let $f'$ be the test statistic induced by $\le$. We have
\begin{align*}
  f'(x) &= \P[\le x] &&\text{by the definition of $f'$,}\\
        &= \P[f\le f(x)] &&\text{by Lemma~\ref{lem:null2},}\\
        &= \hat f(x) &&\text{by the definition of $\hat f$.}
\end{align*}

\smallskip\noindent
(2) Let $f$ be the test statistic induced by a given test order $\le$. By Lemma~\ref{lem:o-s-o}, $f$ induces $\le$. By (1), $f = \hat f$.
\end{proof}

\section{Exact p-values}
\label{sec:exact}

In this section we define exact p-values. A more general notion of p-value will be given in the next section.

\begin{definition}\label{def:p1}
  Given a probabilistic trial $(\Omega,\P)$ furnished with a test order $\le$, the \emph{exact p-value} of an outcome $x$ is the probability $\P[\le x]$.
\end{definition}

Notice the dependence on the test order. Of course, the test order can be given by means a test pyramid or test statistic. Test statistics are especially popular in applications. In this connection, let us give an alternative definition of exact p-values.

\begin{definition}\label{def:p2}
  Given a probabilistic trial $(\Omega,\P)$ furnished with a test statistic $f$, the \emph{exact p-value} of an outcome $x$ is the number $\hat f(x) = \P[f\le f(x)]$.
\end{definition}

Of course the two definitions are consistent.

\begin{lemma}\mbox{}
\begin{enumerate}
\item If a test statistic $f$ induces a test order $\le$ then $\P[\le x] = \P[f\le f(x)]$.
\item If a test order $\le$ induces a test statistic $f$ then $\P[f\le f(x)] = \P[\le x]$.
\end{enumerate}
\end{lemma}

\begin{proof}\mbox{}

\smallskip\noindent
(1) Use Lemma~\ref{lem:null2}, part~1.

\smallskip\noindent
(2) By Lemma~\ref{lem:o-s-o}, $f$ induces $\le$. Now use (1).
\end{proof}

\Q Let's go back to the two examples of \S\ref{sec:logic}, compute the exact p-values of the actual outcomes and compare them.

\A The examples do not have sufficient information. We need to furnish them with reasonable test tools.
For the Coin Example, we can use the test statistic $M(x) = \Min{\{\Heads{(x)},\Tails{(x)}\}}$ defined in the continuation of the Coin Example in \S\ref{sec:tools}.
With respect to that test statistic, the exact p-value $p_1$ of the actual outcome is as follows.
\[
  p_1 = \P[f\le1] = \frac{2(1+42)}{2^{42}} <
  \frac{2\times2^{6}}{2^{42}} = 2^{-35}
\]

Concerning the Lottery Example, one possible approach is given by the following graph $G$.
\begin{itemize}
\item The vertices of $G$ are the 4,000,000 lottery participants plus John, the lottery organizer, whether he bought any lottery tickets or not.
\item Two distinct vertices $X$ and $Y$ are connected by an edge in $G$ if and only if at least one of the following conditions holds:
\begin{itemize}
  \item $X$ is a parent of $Y$ or the other way round.
  \item $X,Y$ are spouses.
  \item $X,Y$ are close friends.
\end{itemize}
\end{itemize}
Given the graph $G$, we have a natural test statistic $\delta(X)$ on the lottery participants, namely the length of the minimal chain of edges from John to $X$. If John bought at least one ticket then the minimal value of $\delta$ is 0; otherwise it is one. Donna, John's wife, is at distance 1 from John, and she bought 4 tickets. We can estimate the exact p-value $p_2$ of the actual outcome from below.
\[
  p_2 = \P(\delta\le1) \ge \frac4{10,000,000} >
  \frac{2^2}{2^{4} \times 2^{10} \times 2^{10}} = 2^{-22}.
\]
It follows that
\[
  p_1 < 2^{-35} < 2^{-22} < p_2.
\]

\Q This is interesting. The win of the lottery organizer's wife seemed more striking to me than getting 41 heads in 42 coin tosses. I wonder whether $p_2$ is small enough to impugn the  hypothesis that the lottery was fair.

\A Well, it seems safe to assume that the lottery participants at distance $\le1$ from John bought less than 10,000 tickets. Under this assumption,
\[
  p_2 < \frac{10,000}{10,000,000} = \frac1{1000} = 0.1\%
\]
which is small enough for the lax as well as strict Cournotians.

\section{p-functions}
\label{sec:pval}

In this section $\eps$ ranges over the non-negative reals.

\begin{lemma}\label{lem:canonic-p}
Let $f$ be a canonic test statistic.
\begin{enumerate}
\item $\P[f\le\eps] =\eps$  if $\eps\in\Range{f}$.
\item $\P[f\le\eps] =\eps$  if $\eps = \sup(S)$ for some $S\subseteq\Range{f}$.
\item $\P[f\le\eps] \le\eps$ for every $\eps$.
\end{enumerate}
\end{lemma}

\begin{proof}
(1) If $\eps = f(x)$ then
\begin{align*}
\P[f\le\eps] &= \P[f\le f(x)] \\
             &= f(x) && \text{as $f$ is canonic} \\
             &= \eps
\end{align*}

\noindent
(2) If $\eps = \sup(S)$ for some $S\subseteq\Range{f}$ then there is a sequence $s_1 < s_2 < \dots$ of elements of $S$ converging to $\eps$. Then
\[
  \P[f\le\eps] = \P[f\le s_1] + \P[s_1< f\le s_2] + \P[s_2<f\le s_3] + \cdots
\]
so that the probability
\[
  \P[f\le s_n] = \P[f\le s_1] + \P[s_1< f\le s_2] + \cdots
                              + \P[s_{n-1}<f\le s_n]
\]
converges to $\P[f\le\eps]$ as $n\to\infty$. Therefore
\begin{align*}
  \P[f\le\eps] &= \lim_{n\to\infty} \P[f\le s_n] \\
               &= \lim_{n\to\infty} s_n &&\text{by (2)}\\
               &= \eps &&\text{by the choice of } s_1 < s_2 < \dots.
\end{align*}

\smallskip\noindent
(3) Let $\eps_0 = \sup\{s\in\Range{f}: s\le\eps\}$. Then $[f\le\eps] = [f\le\eps_0]$ and therefore $\P[f\le\eps] = \P[f\le\eps_0] = \eps_0 \le \eps.$
\end{proof}

\begin{minipage}{\textwidth}
\begin{definition}\label{def:pfun}\mbox{}
\begin{itemize}
\item A \emph{p-function} is a test statistic $f$ such that
$
   \P[f\le\eps]\le\eps
$
for every $\eps$.
\item A p-function $f$ is \emph{exact} if
$
   \P[f\le\eps] = \eps
$
for every $\eps\in\Range{f}$; otherwise $f$ is \emph{conservative}.
\item Values of a p-function are \emph{p-values}.
\end{itemize}
\end{definition}
\end{minipage}

If $f$ is a p-function then $cf$ is a p-function for every $c\ge1$. Indeed
$$
 \P[cf\le\eps] = \P[f\le\eps/c] \le
 \eps/c
 \le \eps.
$$
If $c<1$ then $cf$ may not be a p-function. In particular if $\P[f\le\eps] = \eps$ for at least one $\eps$  then $cf$ is not a p-function because, for that $\eps$, we have
$$\P[cf\le c\eps] = \P[f\le \eps] = \eps > c\eps.$$

\begin{theorem}\label{lem:exact}
For any test statistic $f$, the following two claims are equivalent:
\begin{enumerate}
\item $f$ is canonic.
\item $f$ is an exact p-function.
\end{enumerate}
\end{theorem}

\begin{proof}
The implication (1)$\to$(2) is proven in Lemma~\ref{lem:canonic-p}. To prove the implication (2)$\to$(1), assume that $f$ is an exact p-function, $x$ is an arbitrary outcome and $\eps = f(x)$. We have

\begin{align*}
\P[f\le f(x)] &= \P[f\le \eps] \\
              &= \eps && \text{as $f$ is exact} \\
              &= f(x).&&\qedhere
\end{align*}
\end{proof}

\begin{corollary}\mbox{}
\begin{itemize}
\item Exact p-values are values of a exact p-function.
\item Every test order induces a unique exact p-function.
\item Every test statistic $f$ is equivalent to a unique exact p-function.
\end{itemize}
\end{corollary}

\section{Data snooping}
\label{sec:bonferroni}

\Q What about conservative p-functions? Do you really have any use for them?

\A Yes. A good example is the so-called Bonferroni correction. There is a good explanation in \cite{Abdi}.

\Q But what is the idea?

\A Consider a version of our coin example where, for simplicity, you toss a coin only 15 times. Again the null hypothesis is that all outcomes are equally probable, and again we will use the test statistic $M(x) = \Min\{\Heads(x),\Tails(x)\}$. Suppose that you really want to impugn the null hypothesis. So you repeat the trial over and over until you encounter an outcome where $M$ is 1. At this point, you report that you performed a trial and got an outcome with p-value
\[
  \P[M\leq1] = \frac{2(1+15)}{2^{15}} = \frac{2^5}{2^{15}} = 2^{-10} = \frac1{1024}.
\]

\Q But this is cheating.

\A Exactly. Yet, such data snooping occurs in science though rarely in such a blatant form; a scientist may genuinely forget about some steps in the preliminary analysis of data. In this connection, you may enjoy a relevant issue \cite{xkcd} of the XKCD webcomic.

\Q I remember hearing about publication bias. What is it?

\A It is a rather insidious kind of data snooping. Imagine that a large number of groups of scientists are testing an important null hypothesis. Eventually some group gets a significant p-value and submits their findings for publication while the other groups don't publish anything on the issue. Some influential statisticians argue that most published research findings are wrong and that publication bias may be the main reason for that \cite{Ioannidis:2005}.

\Q You cannot avoid multiple testing of a null hypothesis. What do you do?

\A Consider a case where the same probabilistic trial is performed $n$ times, using test statistics $f_1, \dots, f_n$ and getting p-values $p_1, \dots, p_n$. Of course one would love to report the p-value $\min\{p_1, \dots, p_n\}$ but it needs to be adjusted for the multiple testing of the null hypothesis. The Bonferroni correction is one way to adjust that p-value, by multiplying it by $n$. The p-value $n\cdot\min\{p_1,\ldots,p_n\}$ is valid though usually conservative. More exactly the test statistic $f = n\min(f_1, \dots, f_n)$ is a p-function, typically conservative. Indeed,
$$
  \P[f\le\eps]
  \le
  \P\big(\cup_{i=1}^n[f_i\le\eps/n]\big)
  \le
  \sum_{i=1}^n\P[f_i\le\eps/n]
  \le
  \sum_{i=1}^n\eps/n
  =
  \eps.
$$

\subsection*{Acknowledgments}

We thank Andreas Blass for useful comments on a draft of this paper.


\begin{thebibliography}{99}

\bibitem{Abdi} Herv\'e Abdi, ``Bonferroni and \v{S}id\'ak corrections for multiple comparisons,''\\
    \url{http://www.utdallas.edu/~herve/Abdi-Bonferroni2007-pretty.pdf},\\
    accessed January 8, 2016. Appeared in N.\:J.\:Salkind (ed.), \emph{Encyclopedia of Measurement and Statistics}, Sage Publications, 2007.

\bibitem{CH} D. R. Cox and D. V. Hinkley, ``Theoretical Statistics,'' Chapman \& Hall, 1974.

\bibitem{G206} Yuri Gurevich and Grant O. Passmore, ``Impugning Randomness, Convincingly,'' Bulletin of EATCS 104, June 2011; also --- with an added Prologue --- in Studia Logica 82 (2012), 1--31.

\bibitem{GV} Yuri Gurevich and Vladimir Vovk, ``Fundamentals of p-values,'' unpublished manuscript.

\bibitem{Ioannidis:2005}
  John P. A. Ioannidis,
  ``Why Most Published Research Findings Are False'', PLoS Medicine, Volume 2, Issue 8, e124, August 2005.

\bibitem{xkcd} XKCD, \url{http://xkcd.com/882/}, accessed on Jan.\ 8, 2016.

\end{thebibliography}
\end{document}